\documentclass[12pt]{article}%
   
\usepackage{amsmath,enumerate}
\usepackage{amsfonts}
\usepackage{amssymb}

\usepackage{comment}
\usepackage{color}

\newtheorem{theorem}{Theorem}[section]

\newtheorem{algorithm}[theorem]{Algorithm}

\newtheorem{claim}[theorem]{Claim}

\newtheorem{conjecture}[theorem]{Conjecture}

\newtheorem{proposition}[theorem]{Proposition}

\newenvironment{proof}[1][Proof]{\noindent\textbf{#1.} }
{\hfill \ \rule{0.5em}{0.5em}}

\title{Forbidden subgraphs and complete partitions}
\author{John Byrne\thanks{Department of Mathematical Sciences, University of Delaware. \texttt{jpbyrne@udel.edu}} \and Michael Tait\thanks{Department of Mathematics \& Statistics, Villanova University. Research partially supported by NSF grants DMS-2011553 and DMS-2245556. \texttt{michael.tait@villanova.edu}} \and Craig Timmons\thanks{Department of Mathematics and Statistics, California State University, Sacramento.  Research partially supported by Sac State (P3) Program.  
\texttt{craig.timmons@csus.edu}}}

\begin{document}
\maketitle

\begin{abstract}
    A graph is called an $(r,k)$-graph if its vertex set can be partitioned into $r$ parts, each having at most $k$ vertices and there is at least one edge between any two parts. Let $f(r,H)$ be the minimum $k$ for which there exists an $H$-free $(r,k)$-graph. In this paper we build on the work of Axenovich and Martin, obtaining improved bounds on this function when $H$ is a complete bipartite graph or an even cycle.  Some of these bounds are best possible up to a constant factor
    and confirm a conjecture of Axenovich and Martin in several cases.  
    \end{abstract}

%%%%%%%%%%%%%%%%%%%%%%%%%%%%%%%%%%%%%%%%%%%%%%%%%%%%%%%%%%%%%%%%%%%

%REVISION1  WITH HYPERGRAPHS REMOVED
%%%%%%%%%%%%%%%%%%%%
%%%%%%%%%%%%%%%%%%%%%%%%%%%%%

\section{Introduction}

Let $r \geq 2$ and $k$ be positive integers and 
write $[r]$ for $\{1,2, \dots ,r \}$.
A graph $G$ is an \emph{$(r,k)$-graph} if there is a partition of $V(G)$ into  sets $V_1, V_2, \dots , V_r$ such that $|V_i| \leq k$ for $1 \leq i \leq r$, and for all $1 \leq i < j \leq r$, there is 
an $e \in E(G)$ with one endpoint in $V_i$ and the other in $V_j$.   
Given an $(r,k)$-graph $G$, if each $V_i$ induces a connected subgraph, then contracting each $V_i$ to a single vertex shows that $K_r$ is a minor of $G$. Consequently, an $(r,k)$-graph with partition $V_1,\dots, V_r$ is sometimes called a \textit{$k$-split of $K_r$}.  A more general notion of $k$-splits exists and the study of splits goes back to Heawood \cite{heawood}.    
Originally graph splits were mostly studied with respect to topological properties (see for example \cite{planarsplit}).  More recently, the study of splits has been considered from an extremal perspective. 
There are several interesting questions one can ask, such as the largest complete minor in an $H$-free graph \cite{Fox, KrivSudakov} or in a graph without a sparse cut \cite{KrivNenadov}.  

In another direction, if each $V_i$ is an independent set, 
then the parts $V_1, \dots , V_r$ give a proper vertex coloring of $G$.  This leads to the achromatic number of a graph which we define now.  A \emph{complete $r$-coloring} of a graph $G$ is a proper vertex coloring $c : V(G) \rightarrow [r]$ such that between any two distinct color classes there is at least one edge.  The {\em achromatic number} of $G$, denoted $\chi_a (G)$, is the maximum $r$ for which $G$ has a complete $r$-coloring.  
If there is no restriction on the $V_i$'s, then we have  
a \emph{pseudocomplete $r$-coloring} of $G$.  Pseudocomplete colorings drop the restriction that the coloring must be proper, and we write 
$\psi (G)$ for the {\em pseudoachromatic number} of $G$.  
Complete and pseudocomplete colorings of graphs and their associated parameters have a large body of work devoted to them. Researchers have studied them for specific graph families
\cite{achromaticCirculant, achromaticJGT, achromaticDM,  achromaticDM2, achromaticJCTB2} and
%algorithmic and complexity questions \cite{achromaticJGT2, achromaticJCTB, completeCombinatorica, achromaticSIDMA}), 
hypergraphs \cite{completeDM, completeDMhypergraph}.
The surveys \cite{edwardsSurvey, achromaticSurvey} contain additional references.  

Kostochka \cite{Kostochka} and Thomason \cite{Thomason2006} used pseudocomplete colorings in their influential work on the minimum degree forcing a complete minor, so there are useful relationships between these two notions (see Section 3 of \cite{Thomason2006} or Bollob\'{a}s, Reed, Thomason \cite{Bollobas} for brief discussions).  The focus of this paper is a function,
considered by Barbanera and Ueckerdt \cite{Barb} and studied further by Axenovich and Martin \cite{AxeMartin}, that falls into the branch of extremal partitioning problems with (possibly) some condition on the parts.  
For a graph $H$, define
\[
f ( r, H) = \min \{ k : \mbox{exists an $H$-free $(r,k)$-graph} \}.
\]
Assuming that $H$ has minimum degree at least 3, Axenovich and Martin 
give a simple construction that takes an $H$-free $(r,k)$-graph
and produces an $H$-free graph with a $K_r$ minor by adding at most $k+1$ vertices and at most $2k$ edges to each $V_i$ so that the new color classes are connected.  Regarding coloring, if $G$ is an 
$(r,k)$-graph, then $r \leq \psi (G)$.  Furthermore, any 
complete or pseudocomplete $r$-coloring 
defines an $(r,k)$-partition for some $k$. Since we will consider only $H$-free graphs in this  paper and removing edges cannot create a copy of $H$, any graph with partition $V_1,\dots, V_r$ that we consider may be assumed to have each $V_i$ an independent set.
So bounds for $f(r,H)$ can imply bounds on achromatic numbers and pseudoachromatic numbers and vice versa.
However, in this paper we want to minimize the largest size of a color class in a pseudocomplete $r$-coloring rather than maximizing the number of color classes. Because of these differing objectives, we phrase all of our work in the language of \cite{AxeMartin} instead of using the terminology from these graph coloring and minor problems. 
  It would be interesting to explore these connections in more depth, but that is not the focus of this paper.

For a set $\mathcal{H}$ of graphs, write 
$\textup{ex} ( n , \mathcal{H})$ for the Tur\'{a}n number of $\mathcal{H}$.  This is the maximum number of edges in an $n$-vertex graph that does not contain any graph in $\mathcal{H}$ as a subgraph.  When $\mathcal{H} = \{H \}$, we write $\textup{ex} ( n , H)$ instead of 
$\textup{ex}(n , \mathcal{H} )$.   
A simple observation is that if $G$ is an $H$-free $(r,k)$-graph, then 
\[
\binom{r}{2} \leq e(G) \leq \textup{ex}( rk , H).
\]
The first inequality is true since there is at least 
one edge between any two of the $r$ parts.  The second holds because $G$ is $H$-free with at most $rk$ vertices.  For fixed $r$, this inequality gives a lower bound on $k$ which implies a lower bound on $f(r,H)$.  Any lower bound on $f(r,H)$ obtained using  
\begin{equation}\label{trivial lb}
\binom{r}{2} \leq \textup{ex}( rk, H) 
\end{equation}
will be called the ``trivial (lower) bound".  The inequality $\binom{r}{2} \leq e(G)$ also holds if $r$ is replaced by $\chi_a (G)$, $\psi (G)$, or the Hadwiger number of $G$ which is the largest $r$ such that $K_r$ is a minor of $G$ (see Proposition 1 of \cite{KrivSudakov} for example).      

If $H$ is not bipartite, a simple argument given in \cite{AxeMartin} shows that $f(r,H) \leq 2$ for all $r \geq 1$.  For this reason we will assume that $H$ is bipartite.  
%We will prove a similar generalization to $m$-uniform hypergraphs.  
Supposing for a moment that 
$\textup{ex}(N, H) = \Theta( N^{ \eta} )$ for some
$0 < \eta < 2$, the trivial lower bound 
implies $f(r,H) = \Omega ( r^{2 / \eta - 1} )$.  
Axenovich and Martin proved an upper bound that differs from this lower bound by a logarithmic factor under an assumption on the Tur\'{a}n number of $H$.  

\begin{theorem}[Axenovich, Martin \cite{AxeMartin}]\label{AM:theorem}
	If $H$ is a bipartite graph that is not a forest with $\textup{ex}(N, H) = \Theta( N^{\eta} )$, then there are positive constants $c$ and $C$ such that for all $r$,  
	\[
	c r^{2/ \eta - 1} \leq f(r, H) \leq C r^{2/ \eta - 1} \log^{1/ \eta }r.
	\]
	\end{theorem}
They conjectured that the trivial lower bound (\ref{trivial lb}) gives the correct order of magnitude.  If true, it shows that the logarithmic factor in Theorem \ref{AM:theorem} can be removed.  
       
\begin{conjecture}[Axenovich, Martin \cite{AxeMartin}]\label{conj}
	Let $H$ be a bipartite graph that is not a forest.  There is a positive constant $c = c(H)$ such that for any positive integers $r$ and $k$ for which $\textup{ex}( rk, H) > c \binom{r}{2}$, there is an $H$-free $(r,k)$-graph, i.e., $f(r,H) < k$.      
	\end{conjecture}  

Using the incidence graph of an affine plane, they proved that Conjecture \ref{conj} is true for $C_4$.  

\begin{theorem}[Axenovich, Martin \cite{AxeMartin}]\label{AM:theorem2}
The function $f(r , C_4)$ satisfies 
\[
r^{1/3} - o (r^{1/3} ) 
\leq f( r , C_4) \leq 
2 r^{1/3}  + o ( r^{1/3} ) .
\]
\end{theorem}
While this paper was under review, Taranchuk and the third author \cite{Taranchuk2024} removed the factor of 2 in the upper bound and so 
$f(r,C_4) = (1 + o(1))r^{1/3}$.  Now for any $d \geq 1$, $C_4$ is a subgraph of $K_{2,d+1}$ so 
that $f(r,K_{2,d+1}) \leq f(r,C_4 ) = (1 + o(1))r^{1/3}$.    Combining this with the asymptotic formula 
 $\textup{ex}(N, K_{2,d+1})= \frac{\sqrt{d}}{2} N^{3/2} + o(N^{3/2})$, proved by F\"{u}redi \cite{Furedi}, and 
 the trivial bound (\ref{trivial lb}) gives
  \[
 \frac{1}{d^{1/3}} r^{1/3} - o(r^{1/3} ) 
 \leq f( r , K_{2,d+1}) 
 \leq r^{1/3}  + o (r^{1/3} ).
 \]
 For fixed $d \geq 2$, Theorem \ref{AM:theorem2} proves Conjecture \ref{conj} for $H = K_{2,d+1}$.  However, this result does not capture the dependence of $f(r, K_{2,d+1})$ on $d$.   
 Using the projective norm graphs, we prove an upper bound on $f(r, K_{2,d+1})$ to one that gives the correct dependence on $d$.  Additionally, we confirm Conjecture \ref{conj} for $K_{ t , s}$ whenever $s>(t-1)!$.

\begin{theorem}\label{theorem:ktt}
(a) For $d \geq 3$,  
\[
 \frac{1}{d^{1/3}}   r^{1/3}  - o(r^{1/3}) \leq 
f( r , K_{2,d + 1} ) \leq 
\frac{2}{d^{1/3}} \cdot \frac{1}{( 1 - 1.5 d^{-1/2} )^{5/3} }r^{1/3} + o(r^{1/3}).
\]

\smallskip
\noindent
(b) 
For $t \geq 3$,
\[
\frac{1}{  (  (t-1)! - t )^{ 1/ (2t - 1 ) } } 
r^{1/(2t-1) } - o (r^{ 1 / (2t-1) })
 \leq 
 f( r , K_{ t , (t-1)! + 1  } ) 
 \leq 2 r^{ 1 / (2t-1) } +  o ( r^{ 1/ (2t - 1)  } ).
\]
%c_d r^{1/3} + o (r^{1/3}),
%\]
%where 
%$c_1 = 2$, 
%$c_2 = c_3 = 1.89$, 
%$c_5 = c_4= 1.26$, 
%$c_7  = c_6 = 1.21$, $c_{11} = c_{10} = c_9 = c_8 = 1.20$, and 
%$c_d = \frac{2}{d^{1/3}} \cdot 
%\frac{1}{ ( 1 - \frac{3}{ 2 \sqrt{d} }  )^{5/3} }.$
 \end{theorem}
Unless $d > 34$, the upper bound in part (a) of Theorem \ref{theorem:ktt} does not improve upon the bound   
$f(r,K_{2,d+1}) \leq (1 + o(1))r^{1/3}$.    
 
%It is important to remark that improvements for small $d$ are possible using our proof technique. 
%In particular,  for $2 \leq d \leq 5$
%\begin{center}
%    $f( r , K_{2,4} ) \leq f( r , K_{2,3} ) \lesssim 1.89 r^{1/3}$ ~~ and 
%    ~~
%    $f( r , K_{2,6} ) \leq f( r , K_{2,5} ) \lesssim 1.26 r^{1/3}$.
%\end{center}
%Details for how to obtain these bounds can easily be read out of the proof of Theorem \ref{theorem:ktt} part (a).  We have phrased the theorem without the cases corresponding to smaller $d$ for simplicity and to emphasize the dependence on $d$.  Improvements for $7 \leq d \leq 14$ are listed in Section~\ref{conclusion}.  

Next, we turn our attention towards even cycles by confirming 
Conjecture \ref{conj} for $C_6$ and $C_{10}$.  

\begin{theorem}\label{theorem:cycles}
For the even cycles $C_6$ and $C_{10}$, 
\begin{center}
	$f( r, C_6) \leq 2 r^{1/2} + o ( r^{1/2} ) $~~ and ~~
 $f(r,C_{10}) \leq 2 r^{2/3}  + o( r^{2/3})$.
	\end{center}
\end{theorem}

As in the case of $C_4$ and inspired by our approach, the upper bounds were improved, while this paper was being reviewed, to 
$f(r , C_6) \leq r^{1/2}+ o(r^{1/2}) $ and
$f(r,C_{10} ) \leq r^{2/3} + o(r^{2/3})$ in \cite{Taranchuk2024}.  

The graphs used to prove Theorem \ref{theorem:cycles} were constructed by Wenger \cite{Wenger}.  This family of graphs is not $C_{ 2 \ell}$-free when $\ell \neq 2,3,5$.  
One of the more notable unsolved problems 
in extremal graph theory is to determine the order of magnitude of 
$\textup{ex}(n, C_{ 2 \ell } )$ for $\ell \notin \{2,3,5 \}$.  
With the order of magnitude of $\textup{ex}(N,C_{ 2 \ell } )$ not known in these cases, one cannot apply Theorem \ref{AM:theorem} directly. The authors of \cite{AxeMartin} also proved the following result which relaxes the assumption $\mathrm{ex}(N,H)=\Theta(N^\eta)$.

\begin{theorem}[Axenovich, Martin \cite{AxeMartin}] \label{AM general theorem}
Let $H$ be a bipartite graph that is not a forest and such that for any sufficiently large $N$, $CN^a\le \mathrm{ex}(N,H)\le C'N^b$ for positive constants $C,C'$ and exponents $a,b$ with $b-a<\frac{(2-b)(b-1)}{5-b}.$
\begin{itemize}
    \item If $\mathrm{ex}(nk,H)\ge 12n^2\log n$, then $f(n,H)\le k+o(k)$.
    \item If $\mathrm{ex}(nk,H)<{n\choose 2}$, then $f(n,H)\ge k$.
\end{itemize}
\end{theorem}

The densest constructions of $C_{2\ell}$-free graphs are due to Lazebnik, Ustimenko, and Woldar \cite{LUW}.  They show that
$$\textup{ex}(N,C_{2\ell})=\Omega(N^{1+\frac{2}{3\ell-3+\varepsilon}})$$
where $\varepsilon=0$ if $k$ is odd and $\varepsilon=1$ if $k$ is even. Here the exponent in the best known lower bound is too far from that of the best known upper bound $\textup{ex}(N,C_{2\ell})=O(N^{1+\frac{1}{\ell}})$ to satisfy the assumption $b-a<\frac{(2-b)(b-1)}{5-b}.$ However, in the proof of Theorem \ref{AM general theorem} the only use of this assumption is to obtain an $H$-free graph whose maximum degree is not too large compared to its average degree. Since the graphs in \cite{LUW} are regular, the proof of Theorem \ref{AM general theorem} applies to them giving the following.

\begin{proposition} \label{general even cycles result}
For $\ell\ge 2$, we have $f(r,C_{2\ell})\le Cr^{\frac{3\ell-2}{3\ell+2}}(\log r)^{\frac{3\ell}{3\ell+2}}.$
\end{proposition}

We were unable to find an explicit partition of the $C_{2\ell}$-free graphs which would improve this upper bound above.

\begin{comment}
Nevertheless, we are able to obtain upper bounds on $f(r,C_{2\ell})$ using a different approach which gives an upper bound on $f(r,H)$ whenever 
one has an $H$-free graph with certain pseudorandomness properties quantified by the eigenvalues of the graph (see Theorem \ref{pseudorandom}). Ideally we would then apply this 
theorem to the Lazebnik-Ustimenko-Woldar graphs.  Unfortunately,
with the exception of a few special cases, the eigenvalues of these graphs 
are not well-understood.  While not as dense, more is known regarding
the eigenvalues of the $C_{2 \ell }$-free graphs constructed by Lubotzky, Phillips, and Sarnak \cite{LPS}.   We discuss this more in Section \ref{conclusion}, including how proof of a conjecture of Ustimenko  would improve Theorem \ref{lpsgraphs}.

\begin{theorem}\label{lpsgraphs}
    For $\ell \geq 2$, we have $f(r,C_{2\ell})\le Cr^{\frac{3\ell-2}{3\ell+2}}(\log r)^{\frac{3\ell}{3\ell+2}}$ for infinitely many $r$.
\end{theorem}

The trivial lower bound \eqref{trivial lb} and the even cycle 
theorem  $\mathrm{ex}(N, C_{2\ell}) = O\left(N^{1+1/k}\right)$ imply that $f(r, C_{2\ell}) = \Omega\left( r^{\frac{l-1}{l+1}}\right)$. We conclude that Theorem \ref{theorem:cycles} determines the order of magnitude for $f(r, C_6)$ and $f(r, C_{10})$, while Theorem \ref{lpsgraphs} leaves a gap in the exponents for all $\ell \not\in \{2,3,5\}$. 
\end{comment}

For integers $K , \ell \geq 2$, the \emph{theta graph} $\theta_{K,\ell }$ is the graph consisting of $K \geq 2$ internally disjoint paths of length $\ell$ all having the same two vertices as endpoints.  Observe that $\theta_{2, \ell} = C_{ 2 \ell}$.  The magnitude of $\textup{ex}(n, C_8)$ is unknown, yet Verstra\"{e}te and Williford \cite{VerWilliford} discovered an algebraic construction that gives the lower bound in the asymptotic formula $\textup{ex}(n, \theta_{3,4} ) = \Theta(n^{5/4})$.  Using their construction, we confirm Conjecture \ref{conj} for $\theta_{3,4}$.  

\begin{theorem}\label{th:theta graph}
There are positive constants $c$ and $C$ such that for all $r$, 
\[
c r^{3/5} <  f(r, \theta_{3,4}) < C  r^{3/5}.
\]
\end{theorem}

Finally we turn to bounds for arbitrary forbidden subgraphs. The proof of Theorem \ref{AM:theorem} uses a random partition of the vertex set of an $H$-free graph. Thus, it remains to find deterministic algorithms for constructing $H$-free $(r,k)$-graphs. We describe in Section \ref{Section pseudorandom} such an algorithm which can be used when there exists a dense pseudorandom $H$-free graph.

\begin{theorem} \label{algorithm works}
    Let $H$ be a bipartite graph with no vertices of degree 1. Suppose there exists an $H$-free $d$-regular graph $G$ on $n$ vertices, where $d=n^a$ and $a<\frac{1}{3}$. Let $\rho = \max\{\rho_2, -\rho_n\}$ if $G$ is not bipartite and $\rho = \rho_2$ if $G$ is bipartite and assume that $\rho=O(d^\frac{1}{2})$. Then Algorithm \ref{algorithm description} terminates with a complete partition of a graph $G'$, with $m=\frac{1}{4}n^{\frac{1+a}{2}}/(\log n)^{\frac{1}{2}}$ parts and maximum part size at most $Cm^{\frac{2}{1+a}-1}(\log m)^{\frac{1}{1+a}}$, where $C$ is an absolute constant.
\end{theorem}

If $G$ is optimal in the sense that $\mathrm{ex}(N,H)=\Theta(N^{1+a})$, then this result provides the same upper bound as Theorem \ref{AM:theorem}. Thus, Theorem \ref{algorithm works} can be viewed as a deterministic version of Theorem \ref{AM:theorem}. We believe that the pseudorandomness condition imposed on $G$ is not unreasonably strong, since many constructions of $H$-free graphs satisfy the condition. For example, this applies to the incidence graph of a generalized quadrangle and hexagon for the cases of $H=C_6, C_{10}$ and to the $C_{2\ell}$-free graphs constructed by Lubotzky, Phillips, and Sarnak \cite{LPS}.

\begin{comment}

Next we move on to trees. For a tree $T$ with $t$ edges, Axenovich and Martin \cite{AxeMartin} proved\footnote{There seems to be a small error in \cite{AxeMartin} so the results we restate here are slightly different. In applying Lemma 6 of \cite{AxeMartin}, one chooses $k=\min\{k':n\le r(H;k)-1\}$ which introduces a ceiling function.}
\begin{equation}\label{eq:tree AM}
\frac{r-1}{ 2 ( t - 1) } \leq f(r,T) \leq 2\left\lceil\frac{r}{t-1}\right\rceil-1.
\end{equation}
In the special case that $T$ is a star with $t$ edges, they showed that $\frac{r-1}{t-1} \leq f(r,T) \leq \left\lceil\frac{r}{t-1}\right\rceil$.  The following result improves (\ref{eq:tree AM}) by a factor of $2+o(1)$.

\begin{theorem}\label{th:tree graphs}
Suppose $r$ and $t$ are positive integers for which a 
2-$(r,t,1)$ design exists.  If $H$ is 
any connected graph with more than $t$ vertices, 
then 
\[
f( r, H ) \leq \frac{r-1}{t-1}.
\]
\end{theorem}
\end{comment}

\begin{comment}
\begin{theorem}\label{th:tree graphs}
Suppose $t$ is a positive integer and $H$ is any connected graph with more than $t$ vertices. Then 
\[
f(r,H) \leq (1+o_r(1))\frac{r-1}{t-1}.
\]

\end{theorem}

Note in particular that for $T$ a tree on $t+1$ vertices, if the Erd\H{o}s-S\'os conjecture is true then by \eqref{trivial lb} we have $f(r,T) \geq \frac{r-1}{t-1}$ and so Theorem \ref{th:tree graphs} gives an asymptotically sharp estimate for $f(r,T)$ in this case.  
\end{comment}

%%%%
    In Section \ref{lower bound section}, we give some elementary upper and lower bounds.  In Section \ref{cycles}, we prove our results.    
    Section \ref{conclusion} contains some concluding remarks and open problems.

%%%%%%%%%%%%%%%%%%%%%%%%%%%%%%%%%%%%%%%%%%%%%%%%%%%%%%%%

\section{Elementary Bounds}\label{lower bound section}

% \subsection{Notation}

% \subsection{General Lower and Upper Bounds}

 The first result of this section is an easy consequence of (\ref{trivial lb}).

\begin{proposition}\label{lb prop}
Let $\mathcal{H}$ be a family of graphs with 
 $\textup{ex} ( N , \mathcal{H} ) \leq C N^e$ for some positive constant $C$.  
If there exists an $\mathcal{H}$-free $(r,k)$-graph, then 
\[
\frac{ (r - 1 )^2 }{2} \leq C (rk)^e.
\]
\end{proposition}
\begin{proof}	
Suppose that there exists an $\mathcal{H}$-free  
$(r,k)$-graph $G$.  The number of edges of $G$ 
is at least $\binom{r}{2}$ and is at most $\textup{ex} ( rk , \mathcal{H} )$.  Therefore, 
$\binom{r}{2} \leq e(G) \leq \textup{ex} ( rk , \mathcal{H})$ which implies $\frac{ (r-1)^{2} }{ 2 } \leq C (rk)^e$.  
\end{proof}

\bigskip

As $r$ tends to infinity, Proposition \ref{lb prop} gives the asymptotic lower bound 
\begin{equation}\label{lb on k}
f(r,H) \geq (1 - o_r(1)) \frac{r^{2/e - 1} }{ (2C)^{1/e} }.
\end{equation}   

It is noted in \cite{AxeMartin} that 
\begin{equation}\label{graph case}
    f(r ,H) = 2 ~\mbox{for all non-bipartite $H$ with 
    $r \geq |V(H)|$}.
    \end{equation}
We give provide a proof for completeness.  Given a non-bipartite graph $H$ and an integer $r \geq 2$, let $V_1, \dots , V_r$ be disjoint sets, each containing two vertices. In each part $V_i$, color one vertex red and the other blue.  Let $G$ be any graph obtained by connected all parts using edges that join pairs of vertices of different colors.  The two color classes form independent sets in $G$.  Since $G$ is bipartite and $H$ is not, the graph $G$ is $H$-free.  This shows $f(r,H) \leq 2$ for all $r \geq 2$.  If $r \geq |V(H)|$, then there is no $H$-free $(r,1)$-graph since an $(r,1)$-graph must be $K_r$, and $H$ is a subgraph of $K_r$.  We conclude that $f(r,H) = 2$.  

Because of (\ref{graph case}), one assumes that $\mathcal{H}$ contains at least one bipartite graph when investigating $f(r , \mathcal{H})$.  
  
%%%%%%%%%%%%%%%%%%%%%%%%%%%%%%%%%%%%%%%%%%%%%%%%%%%%%%%%

\section{Proofs of Theorems \ref{theorem:ktt}, \ref{theorem:cycles}, \ref{th:theta graph}, and \ref{algorithm works}.}\label{cycles}

%%%%%%%%%%%%%%%%%%%%%%%
% Divided into three subsections
%%%%%%%%%%%%%%%%%%%%%%%%%%%%%%%%%%%%%%%%%%%%%%%%%%%%%%%%
 
\subsection{Complete bipartite graphs}

In this subsection we begin by proving the following result
which will be used to prove 
Theorem \ref{theorem:ktt}.  

\begin{theorem}\label{general kst}
Suppose that $q$ is an even power of an odd prime and 
that $d$ is a positive integer for which $d | q - 1$.  
For any $t \geq 2$ and positive integers $a\leq h$
for which $ha = \frac{q-1}{d}$, 
\begin{equation}\label{norm graph ub2}
	f( q^{t-1} a , K_{ t , ( t - 1)! d^{t-1} + 1} ) \leq a + h + O( h^{2/5} ).
	\end{equation}
 \end{theorem}
\begin{proof}
Let $q$ be an even power of an odd prime and
$d$ be a positive divisor of $q-1$. 
Let $t \geq 2$ be an integer and $N : \mathbb{F}_{q^{t-1}} \rightarrow \mathbb{F}_{q}$ 
be the $\mathbb{F}_{q}$-norm on $\mathbb{F}_{q^{ t - 1} }$ defined by 
$N(x) = x^{ 1  + q + q^2  + \dots  + q^{  t - 2 }}$.
We write $\mathbb{F}_q^*$ for the group of non-zero elements
of $\mathbb{F}_q$ under multiplication and assume that 
$\theta$ is a generator.  Let $\mathcal{K}_d = \langle \theta^{ (q-1)/d } \rangle$ so that $\mathcal{K}_d$ is the unique subgroup of $\mathbb{F}_q^*$ of order $d$.
Write 
$\mathbb{F}_q^* / \mathcal{K}_d$ for the quotient group
of $\mathbb{F}_q^*$ modulo $\mathcal{K}_d$.  This is also a cyclic group whose order is $\frac{q-1}{d}$.  
We define $\mathcal{B}_d (q,t)$ to be the bipartite graph with parts 
$\mathcal{P} = \mathbb{F}_{q^{t-1} } \times ( \mathbb{F}_q^* / \mathcal{K}_d )$ 
and $\mathcal{L} = \mathbb{F}_{q^{t-1} } \times ( \mathbb{F}_q^* / \mathcal{K}_d )$ where subscripts are used to distinguish vertices in the two parts.  
The vertex $(x , b \mathcal{K}_d)_{ \mathcal{P} } $ 
in $\mathcal{P}$
is adjacent to the vertex 
$(y , c \mathcal{K}_d)_{ \mathcal{L} }$ in $\mathcal{L}$ if and only if 
\[
N(x + y) \in bc \mathcal{K}_d.
\]
  The graph $\mathcal{B}_d (q,t)$ comes from the projective norm graph by taking quotients, in a similar spirit to F\"uredi's construction of $K_{2,t+1}$ free graphs \cite{Furedi}.

\begin{claim}\label{quotient claim}
    The graph $\mathcal{B}_d(q,t)$ is $K_{t,(t-1)!d^{t-1}+1}-$free.
\end{claim}
\begin{proof}[Proof of Claim \ref{quotient claim}]
Following the argument in \cite{ars}, it is 
straightforward to prove that $\mathcal B_1(q,t)$ is $K_{t,(t-1)!+1}$-free.  We briefly outline the argument for completeness.  In this case $d = 1$ so 
$\mathcal{K}_1$ is just the multiplicative identity 
in $\mathbb{F}_q^*$ and the quotient $\mathbb{F}_q^* / \mathcal{K}_1$ is the same as $\mathbb{F}_q^*$.  
Fix $t$ distinct vertices $(x_1, b_1)_{ \mathcal{P} } , \dots , (x_t , b_t )_{ \mathcal{P}}$ in $\mathcal{P}$.  
A common neighbor $(y,c)_{ \mathcal{L}}$ must be a solution to the system 
$$\begin{aligned}
        N(x_1+Y)&=b_1C\\
        &\vdots\\
        N(x_t+Y)&=b_tC.
    \end{aligned}$$
    Using Lemma 4 and the proof of Theorem 5 in \cite{ars}, there are at most $(t-1)!$ solutions to this system.  Hence, there is no $K_{t,(t-1)!+1}$ in 
$\mathcal{B}_1 (q,t)$ with $t$ vertices in $\mathcal{P}$.  By symmetry, there is no $K_{t,(t-1)!+1}$ with 
$t$ vertices in $\mathcal{L}$.  We conclude $\mathcal{B}_1 (q,t)$ is $K_{t , (t-1) ! + 1}$-free.  

Next we prove $\mathcal{B}_d (q,t)$ is 
$K_{t , (t-1)! d^{t-1} + 1}$-free.  Again, we follow \cite{ars} (specifically the argument presented on pages 287 and 288).  As seen in the previous paragraph, there is a symmetry between $\mathcal{P}$ and $\mathcal{L}$ and so we handle both cases at once by omitting the subscripts.    
    Fix $t$ distinct elements $(x_1,b_1\mathcal{K}_d ),\ldots,(x_t,b_t \mathcal{K}_d)\in\mathbb F_{q^{t-1}}\times(\mathbb F_q^*/\mathcal{K}_d)$. Let $(y,c \mathcal{K}_d)$ be a common neighbor of these vertices in $\mathcal B_d(q,t)$. 
    Thus, $N(x_i + y ) \in b_i c \mathcal{K}_d$ for $1 \leq i \leq t$.  As in \cite{ars}, this system can be reduced to a system of $t-1$ equations of the form $N(X_i + Y) \in b_i \mathcal{K}_d$ where $1 \leq i \leq t-1$.  For any choice of $d^{t-1}$ elements from $b_1 \mathcal{K}_d , \dots , b_{t-1} \mathcal{K}_d$, there are at most $(t-1)!$ solutions.  Hence, the total number of solutions is at most $(t-1)!d^{t-1}$.  This implies $(x_1,b_1\mathcal{K}_d ),\ldots,(x_t,b_t \mathcal{K}_d)$ have at most $(t-1)!d^{t-1}$ common neighbors in the graph $\mathcal{B}_d(q,t)$.  This completes the proof of the Claim \ref{quotient claim}.
    \end{proof}

    \bigskip

    With Claim \ref{quotient claim} established, we may now return to the proof of Theorem \ref{general kst}.
Let $a$ and $h$ be positive integers with $a\leq h$ and 
$ha = \frac{q-1}{d}$.  Since $\mathbb{F}_q^* / \mathcal{K}_d$ is a cyclic group with $ha$ elements, we can choose a subgroup $H$ with $|H| = h$.  Let $A$ be a set of coset representatives of $H$ in $\mathbb{F}_q^* /  \mathcal{K}_d$.  Thus, $|H|=h$, $|A|=a$, and
\begin{equation}\label{union eq}
\mathbb{F}_q^* / \mathcal{K}_d = \bigcup_{b \mathcal{K}_d \in A} b \mathcal{K}_d H.
\end{equation}  

Partition $\mathcal{P}$ into the sets 
\[
\mathcal{P}_{x , h_1 \mathcal{K}_d } = \{ ( x, a_1 \mathcal{K}_d h_1 \mathcal{K}_d )_{ \mathcal{P} } : 
a_1 \mathcal{K}_d \in A \}~~ \mbox{where} ~~ x \in \mathbb{F}_{q^{t-1}} ~\mbox{and}~ h_1 \mathcal{K}_d \in H.
\]
Similarly, partition $\mathcal{L}$ into the sets 
\[
\mathcal{L}_{ y , a_2 \mathcal{K}_d} = \{ ( y , a_2 \mathcal{K}_d h_2 \mathcal{K}_d )_{ \mathcal{L} } : h_2 \mathcal{K}_d \in H \}
~~\mbox{where}~~y \in \mathbb{F}_{q^{t-1}}  ~\mbox{and}~ a_2 \mathcal{K}_d \in A.
\]  
%For convenience of notation we write $h_1$ instead of $h_1K_d$ in the subscripts
%of $\mathcal{P}_{x,h_1}$ and likewise, $a_2$ for $a_2 K_d$ in $\mathcal{L}_{y, a_2}$.  
Fix $x,y \in \mathbb{F}_{q^{t-1}}$ and 
$h_1 \mathcal{K}_d \in H$, $a_2 \mathcal{K}_d \in A$.  
There is an edge between 
$\mathcal{P}_{x , h_1\mathcal{K}_d }$ and 
$\mathcal{L}_{ y , a_2 \mathcal{K}_d}$
 if and only if 
\[
N( x + y) \in a_1h_1a_2 h_2 \mathcal{K}_d
\]
for some $a_1 \mathcal{K}_d \in A$ and $h_2 \mathcal{K}_d \in H$.  

First suppose $N(x+y) \neq 0$.  Then $N(x+y)$ is some element 
of $\mathbb{F}_q^*$, say $\alpha$. 
By (\ref{union eq}) there is a $b \mathcal{K}_d \in A$ and $h \mathcal{K}_d \in H$ such that 
$\alpha \mathcal{K}_d = (b \mathcal{K}_d) ( h \mathcal{K}_d) = b h \mathcal{K}_d$.
Choose $a_1 \mathcal{K}_d \in A$ such that 
$a_1 \mathcal{K}_d = ba_2^{-1} \mathcal{K}_d$.
Choose $h_2 \mathcal{K}_d \in H$ such that 
$h_2 \mathcal{K}_d = h_1^{-1} h \mathcal{K}_d$.  We then have that  
\[
\alpha \mathcal{K}_d = bh \mathcal{K}_d =  a_1 a_2 h_1 h_2 \mathcal{K}_d
=
a_1 h_1 a_2 h_2 \mathcal{K}_d.
\]  
Therefore, $N(x +y) \in  \alpha \mathcal{K}_d = a_1 h_1 a_2 h_2 \mathcal{K}_d$ and so there is an edge with one endpoint in $\mathcal{P}_{x,h_1 \mathcal{K}_d}$ and the other in $\mathcal{L}_{y,a_2 \mathcal{K}_d}$.  

Now suppose that $N(x+y)  = 0$.  This holds if and only if $y = - x$.  In this case,  
$\mathcal{B}_d (q,t)$ does not have an edge between the parts $\mathcal{P}_{x , h_1 \mathcal{K}_d }$ and $\mathcal{L}_{ y , a_2 \mathcal{K}_d}$. 
For $x \in \mathbb{F}_{q^{t-1}}$, let $\mathcal{D}_{x}$ be the subgraph induced by 
\[
\left( \bigcup_{ h_1 \mathcal{K}_d \in H } \mathcal{P}_{x,h_1 \mathcal{K}_d} \right) \cup \left( 
\bigcup_{ a_2 \mathcal{K}_d \in A} \mathcal{L}_{-x,a_2 \mathcal{K}_d} 
\right).
\]
The sets in this union are pairwise disjoint 
and $\mathcal{D}_x$ is an $R$-partite graph (inducing no edges) with $R = |H| + |A|$.  
We will add a small number of new vertices and edges  to $\mathcal{D}_x$.  
Add at most $f( R , K_{ t , (t-1)! d^{t-1} + 1 } )$ vertices to each of the $R$ parts such that the subgraph induced by these new vertices is 
$K_{t , (t-1)!d^{t-1} + 1}$-free, and there is an edge between any two of the $R$ parts.  Write $\mathcal{P}_{x,h_1 \mathcal{K}_d}'$ and $\mathcal{L}_{ - x , a_2 \mathcal{K}_d}'$ for these new sets, which contain $\mathcal P_{x,h_1 \mathcal{K}_d}$ and $\mathcal L_{-x,a_2 \mathcal{K}_d}$, respectively.  Observe that 
$| \mathcal{P}_{x,h_1 \mathcal{K}_d}' | \leq 
| \mathcal{P}_{x,h_1 \mathcal{K}_d} | + f( R , K_{ t , (t-1)! d^{t-1} + 1 } )$ and similarly for $\mathcal{L}_{ - x , a_2 \mathcal{K}_d}'$.

Let $\psi:
 \left\{ \mathcal{L}_{y , b \mathcal{K}_d } ' : y \in \mathbb{F}_{q^{t-1} } , b \mathcal{K}_d \in A \right\}
	\rightarrow
		\left\{ \mathcal{P}_{x, h \mathcal{K}_d}' : x \in \mathbb{F}_{q^{t-1} } , h \mathcal{K}_d \in H \right\}
	$ be an arbitrary injection.  Since
 $|A| \leq |H|$, such an injection exists.  
 For each $y \in \mathbb{F}_{q^{t-1}}$ and $b \mathcal{K}_d \in A$, let 
\[
V_{ y,b \mathcal{K}_d} = \psi \left( \mathcal{L}_{y , b \mathcal{K}_d }' \right) \cup  \mathcal{L}_{y,b \mathcal{K}_d}' .
\]
We obtain a graph $G = G(q,t,d,H)$ whose vertex set is the union of the
$V_{y,b \mathcal{K}_d}$'s, which is
a $K_{ t , (t-1)! d^{t-1} + 1 }$-free  
$(r,k)$-graph with 
\begin{center}
$r = q^{t-1} |A| = \frac{ q^{t-1} (q-1) }{ d |H| } $ ~~~~
and~~~~ $k \leq  |A| + |H| + 2 f( |A| + |H|  , K_{t , (t-1)! d^{t-1} + 1} )$.
\end{center}
We then have 
\begin{equation}\label{norm graph ub}
f ( q^{t-1} |A| , K_{t, (t-1)!d^{t-1} + 1 } ) \leq 
 |A| + |H| + 2 f( |A| + |H|  , K_{t, (t-1)!d^{t-1} + 1 }) .
\end{equation}
Using Theorem \ref{AM:theorem} (and the known upper bound on $\mathrm{ex}(n,K_{t,(t-1)!d^{t-1}+1})$),  
\[
2 f(  R ,  K_{t, (t-1)!d^{t-1} + 1 }) \leq 2 C R^{ \frac{1}{2t-1} } \log^{ \frac{t}{2t-1} } R 
\leq C' |H|^{2/5}
\]
which gives an upper bound on the last term in (\ref{norm graph ub}).  
Here we have used $|A| \leq |H|$ and the inequality $\frac{1}{2t-1} \leq \frac{1}{3} < \frac{2}{5}$. 
With this estimate, we have completed 
the proof since $a = |A|$ and $h = |H|$.    
\end{proof} 
 
\bigskip

\noindent
\begin{proof}[Proof of Theorem \ref{theorem:ktt} part (a)]
Let $d \geq 3$ be an integer and $D$ be the unique 
positive integer for which $D(D+1) < d \leq (D+1)(D+2)$.  
Observe that 
\[
(D+1/2)^2  = D(D+1) + \frac{1}{4}  < d \leq (D+1)(D+2) < ( D+ 3/2)^2
\]
where the first inequality holds because $D(D+1) < d$ and $d,D$ are integers.  Thus,
\begin{equation}\label{inequality for d}
\sqrt{d} - \frac{3}{2} < D < \sqrt{d} - \frac{1}{2}.
\end{equation}
Next, since $D$ and $D+1$ are relatively prime, the system $x + 1 \equiv 0 ( \textup{mod} ~D)$ and $x - 1 \equiv 0 ( \textup{mod}~D+1)$ has a unique solution, say $x_0 ( \textup{mod}~D (D+1) )$.  Since $x_0 \equiv D-1 ( \textup{mod}~D)$ and $x_0 \equiv 1 ( \textup{mod}~D+1)$, the least residue $x_0$ is relatively prime with both $D$ and $D+1$.  By Dirichlet's Theorem on Arithmetic Progressions, there are infinitely many primes $p$ with 
$p \equiv x_0\, (\textup{mod}~D(D+1)$).  Let $\mathcal{S}$ be the set of all such primes so that for each $p \in \mathcal{S}$, $p +1 \equiv 0 \,( \textup{mod}~D)$ and $p-1 \equiv 0  \,( \textup{mod}~D+1)$.

By Theorem \ref{general kst} with $t = 2$, $d = D(D+1)$, $a = \frac{p-1}{D+1}$, $h = \frac{p+1}{D}$, and $q = p^2$ where $p \in \mathcal{S}$, we have 
\[
f\left( p^2 \cdot \frac{ p-1}{ D+1}  , K_{2,D(D+1)+1} \right) 
\leq \left( \dfrac{1}{ D} + \dfrac{1}{D+1} \right)
p + O( p^{2/5} )
=
\frac{2D+1}{D(D+1) } p + O ( p^{2/5} ).  
\]
If $R_p := \dfrac{p^2 ( p - 1) }{D+1}$, then using monotonicity and (\ref{inequality for d}),  
\begin{align*}
f ( R_p , K_{2, d +1 } )  \leq 
f ( R_p , K_{2, D (D+1) +1 } ) 
& \leq \frac{2D+1}{D (D+1) }( ( D+1) R_p)^{1/3}     + O  ( R_p^{2/15} ) \\
& = \frac{ 2D+1}{ D ( D+ 1)^{2/3} } R_p^{1/3} + O( R_p^{2/15} ) \\
&< \frac{ 2D+1}{ D^{5/3} } R_p^{1/3} + O( R_p^{2/15} ) \\
& < \frac{ 2 \sqrt{d} } { (  \sqrt{d} - 3/2 )^{5/3} } R_p^{1/3} + O(R_p^{2/15}) \\ 
&=  \frac{2}{d^{1/3}} \cdot 
\frac{1}{ (1 - 1.5d^{-1/2} )^{5/3} }  R_p^{1/3} + O( R_p^{2/15} ). 
\end{align*}

The last step will be to use The Prime Number Theorem in Arithmetic Progressions to go from 
$f(R_p , K_{2,d + 1} ) < \frac{ 2  } { d^{1/3}} \cdot 
\frac{ 1}{ ( 1 - 1.5d^{-1/2} )^{5/3} }    R_p^{1/3} + O(R_p^{2/15})$ for $p \in \mathcal{S}$ to 
\[
f(r  , K_{2, d+1} ) < 
\frac{ 2  } { d^{1/3}} \cdot 
\frac{ 1}{ ( 1 - 1.5d^{-1/2} )^{5/3} }    r^{1/3} + o(r^{1/3}).
\]
For $x > 0$, let $\pi_{ D (D+1) , x_0} (x)$ be the number of primes $p$ such that $p \equiv x_0 ( \textup{mod}~ D(D+1))$ and $p \leq x$.  
Let $\textup{Li}(x) = \int_2^x \frac{dt}{ \log t }$ and $\phi$ denote the Euler phi function.  By Theorem 1.3 in \cite{Bennett}, there are positive constants $x_D$ and $c_D$ depending only on $D$ such that 
\begin{equation} \label{PNT}
\left| \pi_{D(D+1) , x_0 } (x) - \dfrac{ \textup{Li}(x) }{ \phi (D(D+1)) } \right| 
<  \frac{ c_D x}{ ( \log x)^2 }
\end{equation}
for all $x > x_D$. 

Let $0 < \epsilon < 1$.  Let $r$ be a positive integer that we will take sufficiently large in terms of $\epsilon$ and $D$.  Define $f_{r,D}(z) = z^3 - z^2 - r(D+1)$.  
The function $f_{r,D}$ is strictly increasing on $(1, \infty)$, continuous, and has one real root $x_1 = (1 + o_r(1))(r(D+1))^{1/3}$ which tends to infinity with $r$.    
Let $r$ be any integer chosen large enough so that $x_D < x_1 < 2 ( r ( D + 1))^{1/3}$ and 
$\frac{ ( \log x_1 )^2 }{ \log ((1 + \epsilon ) x_1 ) } >
\frac{ 3 c_D D^2 }{ \epsilon}$.
  Then, using (\ref{PNT}) and this last inequality,  
\begin{eqnarray*}
    \pi_{D(D+1) , x_0 } ( ( 1 + \epsilon )x_1 ) - 
    \pi_{D(D+1) , x_0 } ( x_1)  
    & \geq & \dfrac{1}{ \phi (D(D+1)) } \left(
    \int_{x_1}^{ (1 + \epsilon )x_1 } \frac{dt}{ \log t} \right) 
    - \frac{3 c_D x_1}{ ( \log (x_1) )^2 } \\
    & > & 
    \frac{ \epsilon x_1 }{ D^2  \log ((1 + \epsilon) x_1) } 
    - \frac{3c_D x_1}{  ( \log x_1)^2 } > 0.
 \end{eqnarray*}
Let $p$ be a prime with $x_1 \leq p \leq ( 1 + \epsilon )x_1$
and $p \equiv x_0 ( \textup{mod}~D(D+1) )$.  Write $R_p = \frac{p^2(p-1)}{D+1}$ as before.  Since $x_1$ is a root of $f_{r,D}(z)$, we have 
$0 = f_{r,D}(x_1) \leq f_{r,D}(p) \leq f_{r,D}(x_1 + \epsilon x_1)$ which implies 
\[
0 \leq p^3 - p^2 - r (D+1) \leq (3\epsilon+3\epsilon^2+\epsilon^3)x_1^3 - (2\epsilon+\epsilon^2)x_1^2.   
\]
Thus, 
\[
r \leq \frac{ p^2 ( p - 1) }{ D + 1}  
\leq r + \frac{  (3\epsilon+3\epsilon^2+\epsilon^3)x_1^3 - (2\epsilon+\epsilon^2)x_1^2   }{D+1} 
< r + \frac{ 7 \epsilon x_1^3 }{D+1} < r ( 1 + 56 \epsilon)
\]
where we have used $x_1 < 2 ( r (D+1))^{1/3}$ for the last inequality.  We now have that $p$ is a prime in $\mathcal{S}$ with 
$r \leq R_p < r ( 1+ 56 \epsilon)$.  Hence,
\begin{eqnarray*}
    f(r,K_{2,d+1})  & \leq &  f(R_p , K_{2,d+1}) 
< 
\frac{ 2  } { d^{1/3}} \cdot \frac{ 1}{ ( 1 - 1.5d^{-1/2} )^{5/3} }    R_p^{1/3} + O(R_p^{2/15}) \\
& < & 
\frac{ 2  } { d^{1/3}} \cdot \frac{ 1}{ ( 1 - 1.5d^{-1/2} )^{5/3} } r^{1/3} ( 1 + 56 \epsilon)^{1/3} + O (r^{2/15}).
\end{eqnarray*}
As $\epsilon$ can be made arbitrarily small by taking $r$ large, we have 
\[
f(r , K_{2,d+1}) < \frac{2}{d^{1/3}} \cdot 
\frac{1}{  (1 - 1.5 d^{-1/2} )^{5/3} } r^{1/3} + o (r^{1/3}).
\]
\end{proof}

While the proof of Theorem \ref{theorem:ktt} is valid for $d \geq 3$, it does not necessarily produce the best constant.  First, for all $d$ we have the upper bound 
$f(r,K_{2,d+1} ) \leq f(r,C_4) = (1 + o(1))r^{1/3}$.  
For small $d$, one can use the inequality $f\left( p^2 \cdot \frac{ p-1}{ D+1}  , K_{2,D(D+1)+1} \right) 
\leq \left( \dfrac{1}{ D} + \dfrac{1}{D+1} \right)
p + O( p^{2/5} )$ to obtain some improvements.  However, these improvements do not have a matching lower bound and require $D(D+1) < d \leq (D+1)(D+2)$.  The smallest $d$ for which this argument beats $f(r, K_{2,d+1}) \leq (1+ o(1))r^{1/3}$ is $d = 13$ (so $D = 3$) giving
$f(r,K_{2,14}) \leq ( 0.9261 + o(1))r^{1/3}$.  

%For the small cases $d<12$, we improve the constant obtained in Theorem \ref{theorem:ktt} part (a). Roughly speaking, given $d$, we look for the admissible pair $(D_1,D_2)$ such that applying monotonicity
%\[
%f( r , K_{2,d +1} ) \leq 
%f(r , K_{2 , d } ) \leq \cdots 
%\leq
%f(r , K_{2,D_1 D_2 + 1} ) \leq \left( \frac{1}{D_1} 
%+ \frac{1}{D_2} \right) D_2^{1/3} r^{1/3} + o( r^{1/3} )
%\]
%gives the best result. For example, when $d=3$ we note that $(1,2)$ is admissible and so $f(p^2\cdot\frac{p-1}{2},K_{2,4})\le f(p^2\cdot\frac{p-1}{2},K_{2,3})\le(\frac{1}{1}+\frac{1}{2})p+O(p^{2/5})$. Along with density-of-primes arguments similar to the one above, this method is sufficient to obtain the improved constants listed in Section \ref{conclusion}.

\bigskip
\noindent

\begin{proof}[Proof of Theorem \ref{theorem:ktt} part (b)]
Let $d = 1$.  The condition $d | q - 1$ is satisfied
for any $q$.  Taking $a =  \sqrt{q} - 1$ and $h =  \sqrt{q } + 1$ gives
\[
f( q^{t-1} ( \sqrt{q} - 1 ) , K_{t,(t-1)! +1} ) 
\leq 
2 \sqrt{q}  + O_t (q^{1/5} )
\]
for all $t \geq 2$.  As in \cite{AxeMartin}, this upper bound and a density of primes argument proves  
$f( r , K_{ t, (t - 1)! + 1} ) \leq 2 r^{1 / ( 2t - 1 )} +o(r^{1 / ( 2t - 1 )} ) $.  

The lower bound follows from Proposition \ref{lb prop} and the best-known upper bound on $\textup{ex}(N,K_{s,t})$ \cite{Furedi2, Nikiforov}.  
\end{proof}

\begin{comment}
    where since $\delta$ can be taken to be arbitrarily small, we use the $o(r^{1/3})$ to absorb the $O((pD^2)^{2/5})$ and $(1-\delta)$ terms. Now, using \eqref{inequality for d} again, we have 
    \begin{align*}
        f(r, K_{2,d+1}) &\leq \frac{2\sqrt{d}}{D(D+1)}(r(D+1))^{1/3} + o(r^{1/3}) \leq \frac{2\sqrt{d}r^{1/3}}{D^{5/3}} + o(r^{1/3}) \\
        & \leq \frac{2r^{1/3}}{d^{1/3}} \frac{d^{5/6}}{(\sqrt{d} - \frac{3}{2})^{5/3}} + o(r^{1/3}) = \frac{2r^{1/3}}{d^{1/3}}\left(\frac{2\sqrt{d}}{2\sqrt{d} - 3}\right)^{5/3} + o(r^{1/3})\\
        & = \frac{2r^{1/3}}{d^{1/3}} \left(1 + \frac{3}{2\sqrt{d}-3}\right)^{5/3} + o(r^{1/3}).
    \end{align*}
\end{comment}

\bigskip

%%%%%%%%%%%%%%%%%%%%%%%%%%%%%%%%%%%%%%%%%%%%

\subsection{Cycles}\label{cycle subsection}

Using graphs constructed by Wenger we prove Theorem \ref{theorem:cycles}.  Instead of using Wenger's definiton of these graphs, we will use a different set of equations to define adjacency.  The resulting graphs are isomorphic and an explicit isomorphism is given in \cite{Cioaba2014}.
      
Let $q$ be a power of a prime, $M \geq 1$ be an integer, 
$\mathcal{P} = \{ ( p_1, p_2, \dots , p_{M+1} )_{ \mathcal{P}} : p_i \in \mathbb{F}_q \}$, and 
$ \mathcal{L} = \{ ( \ell_1 , \ell_2 , \dots , \ell_{M+1} )_{ \mathcal{L}} : \ell_i \in \mathbb{F}_q \}$.  
Let $W_M(q)$ be the bipartite graph with parts $\mathcal{P}$ and $\mathcal{L}$ where 
$( p_1, p_2, \dots , p_{M+1} )_{ \mathcal{P}} $ 
is adjacent to $( \ell_1 , \ell_2 , \dots , \ell_{M+1} )_{ \mathcal{L}}$ if and only if $\ell_{j+1} + p_{j+1} = l_j p_1 $ for $1 \leq j \leq M$.  The graph 
$W_M (q)$ has $2q^{M+1}$ vertices and is $q$-regular.  For $M \in \{1,2,4 \}$, $W_M(q)$ is $C_{2M+2}$-free.
These graphs were constructed by Wenger \cite{Wenger} and have been studied extensively.  

\bigskip

\begin{proof}[Proof of Theorem \ref{theorem:cycles}]
Let $M=2$ and consider
the bipartite graph $W_2(q)$ which is $C_6$-free.  
Partition $\mathcal{P}$ and 
$\mathcal{L}$ into $q^2$ sets by letting 
\[
\mathcal{P}_{ p_1 , p_3 } = \{ ( p_1 , p_2 , p_3 )_{ \mathcal{P}} : p_2 \in \mathbb{F}_q \}, ~~ p_1 , p_3 \in \mathbb{F}_q,
\]
and 
\[
\mathcal{L}_{ \ell_1 , \ell_2 } = \{  ( \ell_1 , \ell_2 , \ell_3 )_{ \mathcal{L}} : \ell_3 \in \mathbb{F}_q \}, ~~ \ell_1 , \ell_2 \in \mathbb{F}_q.
\]
Fix $p_1, p_3, \ell_1 , \ell_2 \in \mathbb{F}_q$.  
There is exactly one edge between 
$\mathcal{P}_{ p_1 , p_3 }$ and $\mathcal{L}_{ \ell_1 , \ell_2 }$ 
because the two equations $\ell_2 + p_2 = \ell_1 p_1$ and $\ell_3 + p_3 = \ell_2 p_1$ uniquely determine $p_2$ and 
$\ell_3$.  Arbitrarily pair up the 
$\mathcal{P}_{ p_1 , p_3 }$'s with the $\mathcal{L}_{ \ell_1 , \ell_2 } $'s in a 1-to-1 fashion giving $q^2$ parts, each containing $2q$ vertices.  
This partition shows that $W_2(q)$ is a $(q^2 , 2q)$-graph.  
Using a density of primes argument as 
in the proof of Theorem \ref{theorem:ktt} gives 
$f(r,C_6) \leq 2 r^{1/2} + o(r^{1/2})$.  

Next we prove the upper bound on $f(r , C_{10})$.  Let $M = 4$ and consider
$W_4(q)$.  Partition $\mathcal{P}$ and $\mathcal{L}$ into $q^3$ sets by letting
\[
\mathcal{P}_{ p_1 , p_3 , p_5} = \{ ( p_1 , p_2, p_3, p_4,  p_5)_{ \mathcal{P}} : p_2 , p_4 \in \mathbb{F}_q \}, ~~ p_1 , p_3 , p_5 \in \mathbb{F}_q,
\]
and 
\[
\mathcal{L}_{ \ell_1 , \ell_2 , \ell_4 } = \{ ( \ell_1 , \ell_2 , \ell_3 , \ell_4,  \ell_5 )_{ \mathcal{L}} : \ell_3 , \ell_5 \in \mathbb{F}_q \}, 
~~ \ell_1 , \ell_2 , \ell_4 \in \mathbb{F}_q.
\]
Fix $p_1 , p_3 , p_5, \ell_1 , \ell_2 , \ell_4$.
The system of equations  
\begin{eqnarray*}
	\ell_2 + p_2 &  = & \ell_1 p_1 \\
	\ell_3 + p_3 & = & \ell_2 p_1 \\
	\ell_4 + p_4 & = & \ell_3 p_1 \\
	\ell_5 + p_5 & = & \ell_4 p_1.
\end{eqnarray*}
has a unique solution for $p_2$, $\ell_3$, $p_4$, and $\ell_5$.  
This implies that there is exactly one 
edge between $\mathcal{P}_{p_1, p_3, p_5}$ and 
$\mathcal{L}_{ \ell_1 , \ell_2 , \ell_4 }$.  The last part of the argument is 
almost the same as the $C_6$ case and leads to the upper bound 
$f(r, C_{10}) \leq 2 r^{2/3} + o( r^{2/3}) $.
\end{proof}

\medskip

\begin{comment}
The proofs of the previous theorems show that graphs defined algebraically are good candidates to start with. We now show one more example of this phenomenon in the proof of Theorem \ref{th:theta graph} which uses an algebraically defined graph and subfield structure. We believe that similar methods might be useful when considering graphs that forbid cycles or theta graphs with length a multiple of 4. 
\medskip
\end{comment}

\begin{proof}[Proof of Theorem \ref{th:theta graph}]
Let $q$ be an even power of an odd prime and 
$\mu$ be a root of an irreducible quadratic 
over $\mathbb{F}_{ \sqrt{q}}$.  
An element $x_j \in \mathbb{F}_q$ can be written uniquely
as $x_j = x_{1,j} + x_{2,j} \mu$ where and $x_{1,j}, x_{2,j} \in \mathbb{F}_{ \sqrt{q}}$.  
Let $G$ be the bipartite graph with 
parts 
\begin{center}
	$\mathcal{P} = \{ (v_1,v_2,v_3,v_4)_{ \mathcal{P} } : v_j \in \mathbb{F}_q \}$ 
	and 
	$\mathcal{L} = \{ ( w_1 , w_2, w_3 , w_4)_{ \mathcal{L} } 
	: w_j \in \mathbb{F}_q \}$.  
	\end{center}
	Vertices $(v_1,v_2,v_3,v_4)_{ \mathcal{P} }$ and 
$(w_1, w_2, w_3, w_4)_{ \mathcal{L} }$ are adjacent if and only if  
\begin{eqnarray*}
	v_2 + w_2 &  = & v_1 w_1 \\
	v_4 + w_3 & = & v_1^2 w_1 \\
	v_3 + w_4  & = & v_1 w_1^2.
\end{eqnarray*}
This graph is $\theta_{3,4}$-free \cite{VerWilliford}.  
Write each coordinate $v_j$ as $v_j = v_{1,j} + v_{2,j} \mu$ and $w_j = w_{1,j} + w_{2,j} \mu$ 
where $v_{1,j},v_{2,j},w_{1,j},w_{2,j} \in \mathbb{F}_{ \sqrt{q}}$.  
Partition $\mathcal{P}$ into the $q^{5/2}$ sets 
\[ \mathcal{P}_{ v_1 , v_{3,2} , v_4 }
	=
	\{ ( v_1 , v_{2,1} + v_{2,2} \mu , v_{3,1} + v_{3,2} \mu , 
	v_{4,1} + v_{4,2} \mu )_{ \mathcal{P} } : v_{2,1} , v_{2,2} , v_{3,1} \in 
	\mathbb{F}_{ \sqrt{q} } \}
	\]
where $v_1, v_4 \in \mathbb{F}_q$, and $ v_{3,2}  \in 
\mathbb{F}_{ \sqrt{q} }$.  Each of these parts contains $q^{3/2}$ vertices.   
Partition $\mathcal{L}$ into the sets
\[
\mathcal{L}_{ w_1, w_2 , w_{4,1} } = 
\{ 
( w_1 , w_2 , w_{3,1} + w_{3,2} \mu , w_{4,1} + w_{4,2} \mu )_{ \mathcal{L} } 
: w_{3,1} , w_{3,2}, w_{4,2} \in \mathbb{F}_{ \sqrt{q} } \}
\]
where $w_1, w_2 \in \mathbb{F}_q$ and $w_{4,1} \in \mathbb{F}_{ \sqrt{q} }$.  

Fix two parts $\mathcal{P}_{ v_1 , v_{3,2} ,v_4 }$ and
$\mathcal{L}_{ w_1 , w_2 , w_{4,1} }$.  There will be an edge between these two 
parts provided that there exists 
$v_{2,1} , v_{2,2} , v_{3,1}, w_{3,1}, w_{3,2}, w_{4,2} \in \mathbb{F}_{ \sqrt{q}}$ such that 
\begin{eqnarray*}
	v_{2,1} + v_{2,2} \mu + w_{2,1} + w_{2,2} \mu &  = & v_1 w_1 \\
	v_{4,1} + v_{4,2} \mu + w_{3,1} + w_{3,2} \mu & = & v_1^2 w_1 \\
	v_{3,1} + v_{3,2} \mu + w_{4,1} + w_{4,2} \mu  & = & v_1 w_1^2.
\end{eqnarray*}
The first two equations are solved by choosing $v_{2,1} + v_{2,2} \mu$ to 
equal $v_1 w_1 - w_{2,1} - w_{2,2} \mu$, 
and $w_{3,1} + w_{3,2} \mu$ to equal $ v_1^2 w_1 - v_{4,1} - v_{4,2} \mu$.
The last equation can be solved by 
choosing $v_{3,1}$ and $w_{4,2}$ so that 
$v_{3,1}+ w_{4,2} \mu = v_1 w_1^2 - v_{3,2} \mu - w_{4,1}$.
This gives a unique solution to this system of equations and so we have 
an edge between these parts.  
Take an arbitrary pairing of the $\mathcal{P}_{ v_1 , v_{3,2} ,v_4 }$'s
	and $\mathcal{L}_{w_1 , w_2 , w_{4,1} }$'s to obtain a $(p^5,2p^3)$-graph.  Hence, 
 $f( q^{5/2} , \theta_{3,4} ) \leq 2q^{3/2}$ for $q$ an even power of an odd prime.  By a density of 
 primes argument, 
 \[
 f( r , \theta_{3,4} ) \leq 2 r^{3/5} + o ( r^{3/5} ).
 \]
The trivial bound together with 
$\textup{ex}( N, \theta_{3,4} ) = O(N^{5/4})$ proved by Faudree and Simonovits \cite{FS} implies that $f(r, \theta_{3,4}) = \Omega( r^{3/5})$.
 \end{proof}

%%%%%%%%%%%%%%%%%%%%%%%%%%%%%%%%%%%%%%%%%%%%%%%%%%%%%%%%%%%%%%%%%%%%%%%%%%%%%%

\subsection{Partitioning pseudorandom graphs} \label{Section pseudorandom}

 Let $G$ be a graph on $n$ vertices with adjacency matrix $A$. Then $A$ has $n$ real eigenvalues $\rho_1\ge\cdots\ge\rho_n$.  We will make use of the expander-mixing lemma which relates the second largest eigenvalue to the distribution of the edges. This result can be traced back to at least as early as a theorem about designs in Haemers' Ph.D Thesis \cite{Haemers}. For the non-bipartite and bipartite versions stated below in graph-theoretic terms, we refer to \cite{ac} and \cite{De-Winter} respectively. If $U,W\subseteq V(G)$ then we define
 $e(U,W)=|\{(u,w):u\in U,w\in W,u ~\mbox{is adjacent to}~ w\}|$.
 
 \begin{theorem}[Expander-Mixing Lemma] \label{expander mixing lemma} 
 Suppose that $G$ is a $d$-regular graph with $\rho = \max\{\rho_2, -\rho_n\}$. Then for any $U,W\subseteq V(G)$, we have
 $$\left|e(U,W)-\frac{d}{n}|U||W|\right|\le\rho\sqrt{|U||W|}.$$
 \end{theorem}

%We also use the following version for bipartite graphs which can be found for example in Lemma 8 of \cite{De-Winter}.

\begin{theorem}[Bipartite Expander-Mixing Lemma] \label{expander mixing lemma bipartite}
If $G$ is a $d$-regular bipartite graph with parts $X$ and $Y$ where $\rho_2 = \rho$, then for any $U\subseteq X,W\subseteq Y$, we have
$$\left|e(U,W)-\frac{2d}{n}|U||W|\right|\le\rho\sqrt{|U||W|}.$$
\end{theorem}

Before stating our algorithm we fix some notation. Let $V=V(G)$, and if $G$ is bipartite let $(X,Y)$ be a balanced bipartition. For convenience let $A,B$ both denote $V$ if $G$ is not bipartite, and let $A=X$, $B=Y$ if $G$ is bipartite. For $U,W\subseteq V$, let $N(U)$ be the set of vertices not in $U$ which are adjacent to a vertex in $U$. Unless otherwise indicated, all vertices and adjacency relations are within the graph $G$.

\begin{algorithm} \label{algorithm description} Let $H$ and $G$ be graphs satisfying the assumptions of Theorem \ref{algorithm works}.
    \begin{enumerate}
        \item Initialize $V_1=\cdots=V_m=\emptyset$.
        \item For each $i$, add $n^{\frac{1-a}{2}}(\log n)^{\frac{1}{2}}$ vertices in $G$ to $V_i$ such that each vertex added is not already contained in any other $V_j$ and is also at distance at least $3$ from every other vertex in $V_i$. If $G$ is bipartite then all added vertices should also belong to $X$.
        \item For $1\le i\le m$, define $\mathcal S_i=\{j:j\ne i,e(V_i,V_j)=0\}$ and $s_i=|\mathcal S_i|$. If for every $i$ we have $s_i<n^{\frac{1-a}{2}}(\log n)^{\frac{1}{2}}$, proceed to step 4. Otherwise, for each $1\le i\le m$, choose the vertex in $B-(N(V_i)\cup N(N(V_i)))-(V_1\cup\cdots\cup V_m)$ which is adjacent to the maximum number of vertices in $\bigcup_{j\in\mathcal S_i}V_j$ and add it to $V_i$. Return to the beginning of step 3.
        \item For each $i$, for each $j\in\mathcal S_i$, create a new vertex in $V_i$ and make it adjacent to some vertex in $V_j$ which was added to $V_j$ in one of the previous steps.
\end{enumerate}
\end{algorithm}

\begin{proof}[Proof of Theorem \ref{algorithm works}] Is it possible to perform step 2, because at any point during step 2 when we want to add a vertex to $V_i$, we have
$$
\begin{aligned}
    |A-(N(V_i)\cup N(N(V_i)))-(V_1\cup\cdots\cup V_m)|&\ge\frac{n}{2}-d^2|V_i|-m\cdot|V_1|\\ &\ge \frac{n}{2}-n^{2a+\frac{1-a}{2}}(\log n)^{\frac{1}{2}}-\frac{1}{4}n^{\frac{1+a}{2}+\frac{1-a}{2}}.
\end{aligned}$$
The choice of $a$ ensures that both exponents in the right-hand side are at most 1 and the right-hand side is positive for $n$ sufficiently large.

To show it is possible to perform step 3, we track how the number $s_i$ changes each time a vertex is added. 
We assume that $i$ is a value for which the inequality 
$s_i \geq n^{ \frac{1-a}{2} } ( \log n )^{1/2}$
is true.  
Let $s_i(t)$ be the value of $s_i$ after $t$ iterations of step 3 have been performed. Now during any of the first $2an^{\frac{1-a}{2}}(\log n)^{\frac{1}{2}}$ iterations of step 3 we have
$$\begin{aligned}|B-(N(V_i)\cup N(N(V_i)))-(V_1\cup\cdots\cup& V_m)|\ge  \frac{n}{2}-n^{2a}(n^\frac{1-a}{2}(\log n)^{\frac{1}{2}}+2an^{\frac{1-a}{2}}(\log n)^{\frac{1}{2}})\\
&-\frac{1}{4}\frac{n^{\frac{1+a}{2}}}{(\log n)^{\frac{1}{2}}}(n^\frac{1-a}{2}(\log n)^{\frac{1}{2}}+2an^{\frac{1-a}{2}}(\log n)^{\frac{1}{2}}).\end{aligned}$$
The choice of $a$ ensures that the right-hand side is $\Theta(n)$. Set $W=B-(N(V_i)\cup N(N(V_i)))-(V_1\cup\cdots\cup V_m)$. To apply the expander mixing lemma we would like
$$\frac{d\left|\bigcup_{j\in\mathcal S_i(t)}V_j\right||W|}{n}\gg\rho\sqrt{\left|\bigcup_{j\in\mathcal S_i(t)}V_j\right||W|},$$
in other words $\Theta(n)=|W|\gg\frac{\rho^2n^2}{d^2\left|\bigcup_{j\in\mathcal S_i(t)}V_j\right|}$. Noting that 
$\rho= O (d^{\frac{1}{2}})$ 
and $\left|\bigcup_{j\in\mathcal S_i(t)}V_j\right|\ge s_i(t)n^{\frac{1-a}{2}}(\log n)^{\frac{1}{2}}\ge n^{1-a}(\log n)$ we see that this condition is met. Therefore,
$$e\left(W,\bigcup_{j\in\mathcal S_i(t)}V_j\right)\ge\frac{1}{2}\frac{ds_i(t)n^{\frac{1-a}{2}}(\log n)^{\frac{1}{2}}|W|}{n}=\frac{1}{2n^\frac{1-a}{2}}|W|s_i(t)(\log n)^{\frac{1}{2}}.$$
Thus, there exists a vertex in $W$ with at least $\frac{1}{2n^\frac{1-a}{2}}s_i(t)(\log n)^{\frac{1}{2}}$ neighbors in $\left|\bigcup_{j\in\mathcal S_i(t)}V_j\right|$, and so in iteration $t$ of step 3, a vertex is added to $V_i$ which has at least $\frac{1}{2n^{\frac{1-a}{2}}}s_i(t)(\log n)^{\frac{1}{2}}$ neighbors in $\left|\bigcup_{j\in\mathcal S_i(t)}V_j\right|$. Now since each of the sets $V_j$ has all of its vertices at distance at least 3 from each other, each neighbor of the vertex added belongs to a different set $V_j$ (where $j\in\mathcal S_i(t)$). It follows that
$$s_i(t+1)\le s_i(t)-\frac{1}{2n^\frac{1-a}{2}}(\log n)^{\frac{1}{2}}s_i(t).$$
Using $s_i(0)\le m\le n^\frac{1+a}{2}/(\log n)^{\frac{1}{2}}$, we then have that as long as $|W|$ is large enough as above, 
$$s_i(t)\le \frac{n^{\frac{1+a}{2}}}{(\log n)^{\frac{1}{2}}}\left(1-\frac{1}{2n^\frac{1-a}{2}}(\log n)^{\frac{1}{2}}\right)^t.$$
Using the estimate $\log(1-x)\le -x$, we therefore have
$$\log s_i(t)\le\left(\frac{1+a}{2}\log n-\frac{1}{2}\log\log n\right)-t\frac{(\log n)^{\frac{1}{2}}}{2n^{\frac{1-a}{2}}}.$$
Step 3 terminates when for all $i$, $s_i(t) < n^{\frac{1-a}{2}}(\log n)^{\frac{1}{2}}$. Solving for $t$ in the inequality
$$\left(\frac{1+a}{2}\log n-\frac{1}{2}\log\log n\right)-t\frac{(\log n)^{\frac{1}{2}}}{2n^{\frac{1-a}{2}}}\le\log\left(n^{\frac{1-a}{2}}(\log n)^{\frac{1}{2}}\right)$$
we see that this is guaranteed to happen by iteration $t=2a(\log n)^{\frac{1}{2}}n^{\frac{1-a}{2}}.$

In step 4, none of the vertices created creates a copy of $H$, since $H$ has no vertex of degree 1. 

We calculate the maximum size of a part in the resulting $(m,k)$-graph. For each $V_i$, $n^{\frac{1-a}{2}}(\log n)^{\frac{1}{2}}$ vertices were added in step 2, at most $2an^{\frac{1-a}{2}}(\log n)^{\frac{1}{2}}$ were added in step 3, and at most $n^{\frac{1-a}{2}}(\log n)^{\frac{1}{2}}$ were added in step 4. Therefore we have $f(m,H)\le 4n^{\frac{1-a}{2}}(\log n)^{\frac{1}{2}}$. Choose $C$ such that $C\left(\frac{1}{4}\right)^{\frac{2}{1+a}-1}\left(\frac{1}{2}\right)^{\frac{1}{1+a}}\ge 4$. Using the fact that
$$\log m=\frac{1+a}{2}\log n-\log4-\frac{1}{2}\log\log n\ge\frac{1}{2}\log n$$
for $n$ large enough, we find that
$$Cm^{\frac{2}{1+a}-1}(\log m)^{\frac{1}{1+a}}\ge 4n^{\frac{1-a}{2}}(\log n)^{\frac{1}{2}}\ge f(m,H).$$
\end{proof}

\bigskip

\section{Concluding Remarks}\label{conclusion}

\begin{comment}
In the introduction it was mentioned that the upper bound from Theorem \ref{theorem:ktt} (a) can be improved for small $d$.  For $d \geq 2$, write $c_d$ for a value such that $f(r,K_{2,d+1}) < c_d r^{1/3} + o (r^{1/3})$.  One can use the proof of Theorem \ref{theorem:ktt} to show that $c_3 \leq c_2 \leq 1.89$, $c_5 \leq c_4 \leq 1.26$, $c_7 \leq c_6 \leq 1.21$, $c_{11} \leq c_{10} \leq c_9 \leq c_8 \leq 1.20$, and $c_{14} \leq c_{13} \leq c_{12} \leq 0.93$.  These bounds are much better than that of Theorem \ref{theorem:ktt} for small $d$.  
\end{comment}

In light of Theorem \ref{AM:theorem2}, it would be interesting to determine an asymptotic formula for $f(r , K_{2,d+1})$.  It seems possible that $f(r, K_{2,d+1}) = (r/d)^{1/3} + o(r^{1/3})$ for all $d \geq 1$.  This is true for $d= 1$ by the result of \cite{Taranchuk2024}.     

We proved that $f(r, \theta_{3,4}) = \Theta ( r^{3/5})$.
It was proved by Conlon \cite{Conlon} that 
for every $\ell \geq 2$, there is a $K$ such that 
$\textup{ex}(N , \theta_{K , \ell} ) = \Omega_\ell ( N^{1 + 1/\ell})$, and an upper bound that matches in order of magnitude is in \cite{FS}.  It would be interesting to determine the order of magnitude for $f( r, \theta_{t , \ell})$ for $t\geq K$ and in such cases where it can be done it would further be interesting to determine the dependence on $t$.%Perhaps one could obtain an asymptotic formula for $f( r, \theta_{K , \ell})$.

\bibliographystyle{plain}
	\bibliography{references.bib}
	
\end{document}